\pdfoutput=1
\RequirePackage{fix-cm}
\documentclass[12pt,letterpaper,onecolumn]{article}

\usepackage[left=1.25in,right=1.25in,top=1in,bottom=1in]{geometry}
\usepackage{amsmath,amssymb}
\usepackage{amsthm}
\usepackage{graphicx}
\usepackage{cite}
\usepackage[fontsize=12.5pt]{fontsize}
\usepackage{setspace}
\usepackage{placeins}
\usepackage{needspace}

\newcommand{\R}{\mathbb{R}}
\newcommand{\E}{\mathbb{E}}
\newcommand{\Prob}{\mathbb{P}}
\newcommand{\btheta}{\boldsymbol{\theta}}
\newcommand{\thh}{\hat{\btheta}}
\newcommand{\Bmat}{\mathbf{B}}
\newcommand{\ybar}{\bar{y}}
\newcommand{\epb}{\bar{\epsilon}}

\newtheoremstyle{romanbody}
  {\topsep}{\topsep}{\normalfont}{}
  {\itshape}{:}{0.5em}{}
\theoremstyle{definition}
\newtheorem{definition}{Definition}
\theoremstyle{romanbody}
\newtheorem{theorem}{Theorem}
\newcounter{assumption}
\renewcommand{\theassumption}{A\arabic{assumption}}
\newenvironment{standingassumptions}{
  \begin{list}{\textbf{\theassumption}}{
    \usecounter{assumption}
    \settowidth{\labelwidth}{\textbf{A3}}
    \setlength{\labelsep}{0.6em}
    \setlength{\leftmargin}{\labelwidth}
    \addtolength{\leftmargin}{\labelsep}
    \setlength{\itemsep}{0pt}
    \setlength{\parsep}{0pt}
  }
}{\end{list}}

\title{Blind Random Search with Noisy Loss Measurements:\\
Averaging, Thresholding, and Almost Sure Convergence}

\author{Zixian Zhou and Xintong Jiang}
\date{}

\begin{document}

\maketitle
\pagestyle{plain}

\begin{abstract}
Blind random search repeatedly draws a candidate point and replaces the
current estimate whenever the candidate has a lower loss. In the
absence of noise, the true loss is observed directly. It decreases
strictly at every accepted update and is monotone nonincreasing over all
iterations. Measurement noise can make a worse candidate appear better
and thereby break this monotonicity. To
recover almost sure convergence under noise, we incorporate averaging and
thresholding into the original decision criterion. These two classical
tools are coupled. As the sample sizes grow, the positive threshold
shrinks at a matched rate. These
modifications allow blind random search to recover eventual monotonicity
of the true loss under noisy measurements and to converge almost surely.
\end{abstract}

\section{Introduction}

In many practical optimization problems, direct gradient measurements are not
naturally available, while measurements of the objective function are generally
noisy. Obtaining gradients from system outputs can also require detailed
knowledge of the relationships between system inputs and outputs, whereas
methods based on function measurements can avoid such a detailed model
\cite[pp.~150--153]{spall2003}\cite{spall1994}. These methods appear both in
practical applications and in broader black box optimization procedures. For
example, Chin et al. considered traffic signal control \cite{chin1999}, while
Regis considered deterministic constrained black box optimization
\cite{regis2022}. Among methods based on function measurements, this paper
focuses on blind random search. The method independently samples candidate
points over $\Theta$ without adapting its sampling strategy to information
obtained earlier in the search \cite[pp.~37--38]{spall2003}. This
makes blind random search useful when little exploitable information about the
objective is available \cite[pp.~19--21]{solis1981}.

In the absence of noise, blind random search
repeatedly samples candidates throughout the domain and accepts one only when
its loss is below the current loss \cite{spall2003}. Consequently, the sequence
of loss values at the current estimates, $\{L(\thh_k)\}$, is nonincreasing.
In practice, loss measurements are often corrupted by noise
\cite{spall1994}. We write an observation as
$y(\btheta)=L(\btheta)+\epsilon$, where $\epsilon$ denotes measurement noise
and $\E[\epsilon]=0$.
Under noise, an apparent improvement $y(\thh_{k+1})<y(\thh_k)$ may conceal a
true deterioration
$L(\thh_{k+1})>L(\thh_k)$ \cite{alexander2005}. The algorithm can therefore
accept a candidate with higher true loss, and the monotonicity of the noiseless
search is lost.

Averaging and thresholding are two common mechanisms for mitigating noise in
comparison based search \cite{spall2003}. Averaging reduces the variance of
loss estimates, whereas thresholding reduces false acceptances caused by noise
by requiring the observed improvement to exceed a positive margin
\cite{spall2003}; convergence has been established for several methods using
one or both mechanisms \cite{devroye1976,alexander2005,devroye2002}. To ensure
reliable comparisons, these methods remeasure the loss at the current estimate
for every comparison, so that the observations used in each comparison are
independent of the past selection history \cite{devroye1976}. Such
remeasurement, however, increases measurement cost. Devroye considered
discarding data after a comparison wasteful \cite[p.~706]{devroye1976b} and
worked with two methods that store historical search points and their associated
observations \cite{devroye1976b,devroye2002}.

We share Devroye's view that observations are valuable. Our algorithm combines
increasing averaging with a vanishing positive threshold and reuses accumulated
measurements at the current estimate. Consequently, the observations used in a
comparison are no longer independent of the past selection history. Controlling
that bias is a central contribution of the proof. We show that the true loss is
eventually nonincreasing and that the search converges almost surely to the
global minimizer.

\section{The Algorithm}
\label{sec:algorithm}

Consider the problem of minimizing a loss function
$L(\btheta)$ over a domain
$\Theta \subseteq \R^p$, where only noisy measurements
$y(\btheta) = L(\btheta) + \epsilon$ are available. The following
modification of blind random search \cite{spall2003} combines the
increasing amount of averaging with the dynamic threshold.

\begin{enumerate}
\item \begingroup\emergencystretch=.75em
      \textbf{Initialization:} Fix an integer $k_0 \ge 1$ and select an initial estimate $\thh_0 \in \Theta$. Compute the
      initial averaged loss over $k_0$ observations,
      $\ybar(\thh_0) = \frac{1}{k_0}\sum_{i=1}^{k_0} y_i(\thh_0)$,
      and set $k = k_0$ with $\thh_{k_0} := \thh_0$ and
      $\ybar(\thh_{k_0}) := \ybar(\thh_0)$ (so the search step below
      is defined at $k = k_0$).\par\endgroup
\item \textbf{Search step (iteration $k$):}
  \begin{itemize}
    \item Randomly generate a candidate point
          $\btheta_{\text{new}}(k+1) \in \Theta$. From $k+1$ fresh
          observations, compute the averaged loss value at the candidate
          point,
          \[
            \ybar(\btheta_{\text{new}}(k+1))
              = \frac{1}{k+1}\sum_{i=1}^{k+1}
                    y_i(\btheta_{\text{new}}(k+1)).
          \]
    \item Accept the candidate if
          \begin{equation}
            \ybar(\btheta_{\text{new}}(k+1)) < \ybar(\thh_k) - \tau_k
            \label{eq:threshold}
          \end{equation}
          where the dynamic threshold is
          \[
            \tau_k := c\sigma\sqrt{\frac{\log(k+1)}{k}},
            \qquad c>2\sqrt{2},
          \]
          where $\sigma$ is the known subgaussian scale parameter
          formalized in Assumption~\ref{ass:noise}.
    \item \emph{If accepted:} set
          $\thh_{k+1} = \btheta_{\text{new}}(k+1)$ and
          $\ybar(\thh_{k+1}) = \ybar(\btheta_{\text{new}}(k+1))$.
    \item \emph{If rejected:} set $\thh_{k+1} = \thh_k$.
          \par\noindent
          Draw one additional fresh observation $y_{k+1}(\thh_k)$
          and update the running average,
          \[
            \ybar(\thh_{k+1})
              = \frac{k\,\ybar(\thh_k) + y_{k+1}(\thh_k)}{k+1}.
          \]
    \item At the end of iteration $k$, $\ybar(\thh_{k+1})$ is the
          arithmetic mean of exactly $k+1$ observations at
          $\thh_{k+1}$. Increase $k$ by one and return to the search
          step.
  \end{itemize}
\item \textbf{Termination:} Stop when a predefined criterion is met,
      such as a maximum number of noisy loss measurements, or when the user
      is otherwise satisfied with the current estimate.
\end{enumerate}

\noindent\emph{Relation to prior search rules under noise.}\par
\noindent Devroye's statistical search decision rule recomputes independent
sample means at the current estimate and the candidate point in each noisy
comparison; earlier measurements at the current estimate are not pooled into
the loss estimate used in the next comparison
\cite[p.~47]{devroye1976}. ISSARS \cite[Sec.~3.2]{alexander2005}
compares new sample means at both points. Theorems~5 and~6 of
Devroye and Krzy\.{z}ak
\cite[Thms.~5--6]{devroye2002} instead use $Z_n$ formed from two new
empirical distribution functions. Equation~\eqref{eq:threshold}
compares the averaged loss value at the candidate point with the pooled
loss estimate at the current estimate. Devroye's learning-memory method may
store historical search points and their associated observation summaries
\cite[Sec.~III]{devroye1976b}, whereas the nonsequential strategy analyzed by
Devroye and Krzy\.{z}ak retains every sampled point and repeatedly augments its
observations \cite[Sec.~10]{devroye2002}. Our method instead stores only the
current estimate and its pooled loss estimate, adds one new observation after
every rejection, and couples growing sample sizes to a vanishing threshold,
creating the selection bias studied in Section~\ref{sec:analysis}.

\section{Convergence Analysis}
\label{sec:analysis}
Let $L(\btheta)$ be the loss function and $k$ the iteration index of
the algorithm of Section~\ref{sec:algorithm}. We first recall the
subgaussian noise model used in the convergence analysis.

\begin{definition}[Subgaussian random variable]
A random variable $X$ with mean zero is subgaussian with variance proxy
$\sigma^2$ if
\[
  \E[\exp(\lambda X)]
    \le \exp\bigl(\lambda^2\sigma^2/2\bigr),
  \qquad \lambda\in\R.
\]
\end{definition}
This condition implies
$\Prob(|X|>t)\le 2\exp(-t^2/(2\sigma^2))$ for $t>0$
\cite[Def.~2.2 and Eq.~(2.9)]{wainwright2019}.
It also implies $\operatorname{Var}(X)\le\sigma^2$; the variance proxy
need not equal the actual variance. Examples include centered Gaussian, centered
Bernoulli, and scaled Rademacher distributions, as well as bounded
distributions with mean zero. For a Student's $t$ distribution with finite
degrees of freedom, the tails decrease only as powers of $t$. Such a
distribution cannot satisfy the preceding subgaussian tail bound and is
therefore not subgaussian.

The convergence result is established under the following standing
assumptions.
\begin{standingassumptions}
\item\label{ass:separation} \emph{(Separation of the minimizer).}
      $\btheta^*$ is a minimizer of
      $L$ on $\Theta$ and, for every $\eta > 0$,
      \[
        \inf_{\btheta \in \Theta,\; \|\btheta-\btheta^*\| \ge \eta}
          L(\btheta) \;>\; L(\btheta^*).
      \]
      This condition implies that $\btheta^*$ is the unique minimizer.
\item\label{ass:noise} \emph{(Noise).}
      Each observation has the form
      $y(\btheta) = L(\btheta) + \epsilon$, where all noise variables
      drawn by the algorithm, across all iterations and across both
      the current estimate's and the candidate's observations, are
      mutually independent and identically distributed, have mean
      zero, and are subgaussian with a common variance proxy $\sigma^2$.
\item\label{ass:reachability} \emph{(Reachability).}
      For every $\eta > 0$ and every $k \ge k_0$,
      \[
        \Prob\bigl(
            L(\btheta_{\text{new}}(k+1)) < L(\btheta^*) + \eta
          \bigr) \;\ge\; \delta(\eta)
      \]
      for some $\delta(\eta) > 0$ independent of $k$. Candidate draws
      are independent across iterations and independent of the
      measurement noise, and their sampling distributions are fixed
      in advance as functions of $k$.
\end{standingassumptions}

\begin{theorem}
\label{thm:main}
Suppose that Assumptions~\ref{ass:separation}--\ref{ass:reachability}
hold. Then the sequence $\{\thh_k\}$ generated by the algorithm satisfies
$L(\thh_k) \to L(\btheta^*)$, $\ybar(\thh_k) \to L(\btheta^*)$, and
$\thh_k \to \btheta^*$, all almost surely as $k \to \infty$.
\end{theorem}

\begin{proof}
\renewcommand{\qedsymbol}{}
The proof has eight parts. Parts~I--IV control the comparison noise
uniformly and show that the current noise average converges to zero,
although the observations entering that average depend on earlier
acceptance decisions. Part~V establishes eventual monotonicity of the
true loss; Parts~VI and~VII show that candidates near the optimum recur
and are eventually accepted; Part~VIII adapts the sample path argument
of \cite[Thm.~2.1]{spall2003}. All statements below hold almost surely
unless otherwise noted.

\subsubsection*{Part I: Comparison noise and selection bias}
Under the cumulative averaging rule, at iteration $k$ the
average $\ybar(\thh_k)$ is the arithmetic mean of the $k$ observations
in the accumulated sample. Write
\begin{gather*}
  \ybar(\thh_k) = L(\thh_k) + \epb_k^{\mathrm{cur}}, \\
  \ybar(\btheta_{\mathrm{new}}(k+1))
    = L(\btheta_{\mathrm{new}}(k+1)) + \epb_k^{\mathrm{new}},
\end{gather*}
{\widowpenalty=10000
where $\epb_k^{\mathrm{cur}}$ is the average noise in the $k$
observations used to compute $\ybar(\thh_k)$, whereas
$\epb_k^{\mathrm{new}}$ is the average noise in the $k+1$ new
observations used to compute
$\ybar(\btheta_{\mathrm{new}}(k+1))$. Define the difference between the
two noise averages and the increment in the true loss by\par}
\[
  D_k := \epb_k^{\mathrm{new}}-\epb_k^{\mathrm{cur}},
  \qquad
  \Delta_k := L(\btheta_{\mathrm{new}}(k+1))-L(\thh_k).
\]
The acceptance test is equivalently
\[
  \Delta_k+D_k< -\tau_k.
\]

Under Assumption~\ref{ass:noise}, each newly drawn measurement error is
independent, identically distributed, and has mean zero. Conditional on
the search history, $\epb_k^{\mathrm{new}}$ is therefore centered and
subgaussian with variance proxy $\sigma^2/(k+1)$ because it averages
$k+1$ new measurements \cite[p.~24]{wainwright2019}. The pooled average
$\epb_k^{\mathrm{cur}}$ does not follow the centered distribution of a
fresh independent average: it includes measurements from the most
recently accepted point and is reused rather than remeasured at every
comparison.

The acceptance rule is
\[
  \text{accept}
  \quad\Longleftrightarrow\quad
  \Delta_k+\epb_k^{\mathrm{new}}-\epb_k^{\mathrm{cur}}< -\tau_k.
\]
The decision therefore combines a change in the true loss with the observed
noise: a smaller average of the candidate measurement errors makes acceptance
more likely. Before the decision, $\E[\epb_k^{\mathrm{new}}]=0$. Whenever
$0<\Prob(\text{accept})<1$, conditioning on acceptance gives
\[
  \E[\epb_k^{\mathrm{new}}\mid\text{accept}]<0.
\]
Thus, an accepted candidate average has negative conditional mean even though
every underlying error has mean zero. Upon acceptance,
$\epb_{k+1}^{\mathrm{cur}}=\epb_k^{\mathrm{new}}$, and the $k+1$ candidate
measurements initialize the average at the new current estimate. Hence the
resulting current noise average is not centered and does not follow the
distribution of an ordinary average of independent measurement errors.

Because of this selection bias, we study $\epb_k^{\mathrm{cur}}$ by
separating the sample present when the current estimate was established
from the measurements added after subsequent rejections, as illustrated
in Fig.~\ref{fig:sample-flow}.

\begin{figure}[!t]
\centering
\includegraphics[width=0.86\linewidth]
  {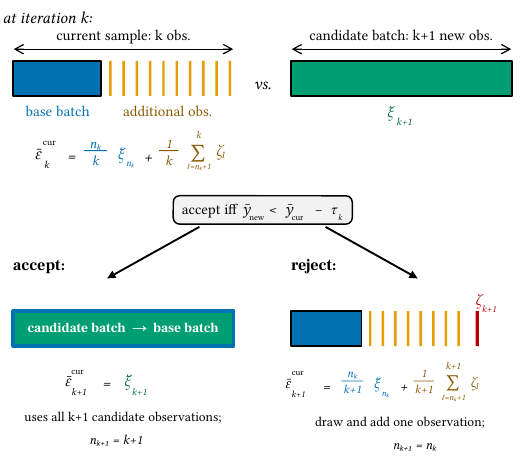}
\caption{Decomposition at iteration $k$: accumulated measurements at
the current estimate (blue and yellow), new candidate measurements
(green), and the measurement added after a rejection (red).}
\label{fig:sample-flow}
\end{figure}

\subsubsection*{Part II: Realized sample decomposition}
We now represent $\epb_k^{\mathrm{cur}}$ by tracing how the algorithm
forms the accumulated sample. Let $n_k$ be the sample size when the
current estimate was established. Before the first acceptance,
$n_k=k_0$. If the most recent acceptance occurred at iteration $j$,
then $n_k=j+1$.

Let $\{\xi_{k_0,i}\}_{i=1}^{k_0}$ be the initialization noises. For
$n\geq k_0+1$, let $\xi_{n,i}$ be the noise in the $i$th of the $n$
measurements taken at the candidate point in iteration $n-1$, and set
$\bar\xi_n=n^{-1}\sum_{i=1}^{n}\xi_{n,i}$. After the first acceptance,
$\bar\xi_{n_k}$ is the average noise from the most recently accepted
candidate. If the candidate at iteration $\ell-1$ is rejected, let
$\zeta_\ell$ be the noise in the additional loss measurement drawn at
the current estimate. With this notation,
\begin{equation}
  \epb_k^{\mathrm{cur}}
    = \frac{n_k}{k}\,\bar\xi_{n_k}
    + \frac{1}{k}\sum_{\ell=n_k+1}^{k}\zeta_\ell.
  \label{eq:cur-decomp}
\end{equation}
The first term represents the measurements used to establish the
current estimate: the initialization measurements before the first
acceptance, or the measurements at the most recently accepted candidate
thereafter. Their average is $\bar\xi_{n_k}$, and their weight in the
current average of $k$ measurements is $n_k/k$. The second term
represents one additional measurement at the current estimate for each
later rejection. There are $k-n_k$ such noise terms, each with weight
$1/k$. The new candidate average is
$\epb_k^{\mathrm{new}}=\bar\xi_{k+1}$.

Figure~\ref{fig:sample-flow} illustrates the same construction. Blue
denotes these $n_k$ measurements, and yellow denotes the $k-n_k$
measurements added after later rejections. Green denotes the $k+1$ new
measurements at the candidate point. If the candidate is accepted, the
green measurements form the accumulated sample at iteration $k+1$, so
$\epb_{k+1}^{\mathrm{cur}}=\epb_k^{\mathrm{new}}$ and
$n_{k+1}=k+1$. If the candidate is rejected, the blue and yellow
measurements remain in the accumulated sample, and one new measurement
is taken at $\thh_k$. Its noise is the red $\zeta_{k+1}$, and
$n_{k+1}=n_k$.

All $\xi_{n,i}$ and $\zeta_\ell$ are i.i.d.\ subgaussian variables
with mean zero and variance proxy $\sigma^2$; the notation distinguishes
noise in the candidate measurements from noise in the additional
measurements.

\subsubsection*{Part III: Comparison bound for a deterministic history}
This decomposition reduces the distributional analysis of
$\epb_k^{\mathrm{cur}}$ to a family of averages indexed by deterministic
$n$. For each $k_0\le n\le k$, define the possible current noise average
\begin{equation}
  A_k(n)
    := \frac{1}{k}\left(
         \sum_{i=1}^{n}\xi_{n,i}
         + \sum_{\ell=n+1}^{k}\zeta_\ell
       \right),
  \label{eq:possible-current}
\end{equation}
where the second sum is zero when $n=k$. Equation~\eqref{eq:cur-decomp}
then gives
\[
  \epb_k^{\mathrm{cur}}=A_k(n_k),
  \qquad
  \epb_k^{\mathrm{new}}=\bar\xi_{k+1}.
\]
For fixed $(n,k)$, $A_k(n)$ is the average of $k$ independent noises
and is subgaussian with variance proxy $\sigma^2/k$. Because acceptance
decisions make $n_k$ random, we first control $A_k(n)$ uniformly over
deterministic $n$ and then substitute $n_k$. The proof of Theorem~1 in
\cite[pp.~709--710]{devroye1976b} uses a related argument: Lemma~3
bounds all later sample sizes before the result is applied to points
generated by the search.

This decomposition is used in Parts~III and~IV to establish the
following two claims. Almost surely,
\[
  \begin{gathered}
    |D_k|
    =\bigl|\epb_k^{\mathrm{new}}-\epb_k^{\mathrm{cur}}\bigr|
    \le\tau_k
    \quad\text{for all sufficiently large }k,\\
    \epb_k^{\mathrm{cur}}\to0.
  \end{gathered}
\]

\emph{Uniform control of the comparison noise.}
For each deterministic $k_0\le n\le k$, define
\[
  D_k(n):=\bar\xi_{k+1}-A_k(n).
\]
The fresh errors
$\xi_{k+1,1},\ldots,\allowbreak\xi_{k+1,k+1}$ are independent of the earlier
errors $\xi_{n,1},\ldots,\allowbreak\xi_{n,n},\allowbreak
\zeta_{n+1},\ldots,\allowbreak\zeta_k$ that form
$A_k(n)$. Consequently, $\bar\xi_{k+1}$ is independent of $A_k(n)$.
Hence, for every $\lambda\in\R$,
\begin{align*}
  \E\!\left[\exp\{\lambda D_k(n)\}\right]
  &= \E\!\left[\exp\{\lambda\bar\xi_{k+1}\}\right]
     \E\!\left[\exp\{-\lambda A_k(n)\}\right] \\
  &\le
     \exp\left(\frac{\lambda^2\sigma^2}{2(k+1)}\right)
     \exp\left(\frac{\lambda^2\sigma^2}{2k}\right) \\
  &\le \exp\left(\frac{\lambda^2\sigma^2}{k}\right).
\end{align*}
{\widowpenalty=10000
Thus $D_k(n)$ is subgaussian with variance proxy $2\sigma^2/k$.
Applying the tail bound to both sides gives\par}
\begin{align*}
  \Prob\bigl(|D_k(n)|>\tau_k\bigr)
  &\le 2\exp\left(-\frac{k\tau_k^2}{4\sigma^2}\right) \\
  &= 2(k+1)^{-c^2/4}.
\end{align*}

\subsubsection*{Part IV: Uniform control of noise with probability one}
There are at most $k$ possible values of $n$ at iteration $k$.
Consequently,
\[
  \Prob\bigl(\exists\,k_0\le n\le k:
       |D_k(n)|>\tau_k\bigr)
  \le 2k(k+1)^{-c^2/4}.
\]
Because $c>2\sqrt{2}$, we have $c^2/4>2$, so the bound on the right is
summable in $k$. The first Borel--Cantelli lemma
\mbox{\cite[p.~188]{feller1957}}
implies that only finitely many of these events occur almost surely, so
their indices form a finite exceptional set.
\begin{samepage}
On this event of probability one, let $J_0(\omega)$ be the first index
beyond this set (or $k_0$ if the set is empty). After substituting the
value $n_k$ selected by the algorithm, the following holds for every
$k\ge J_0(\omega)$:
\begin{equation}
  |D_k|=|D_k(n_k)|\le\tau_k.
  \label{eq:comparison-final}
\end{equation}
\end{samepage}

\emph{Convergence of the current noise average.}
The fresh candidate average $\epb_k^{\mathrm{new}}=\bar\xi_{k+1}$ is
subgaussian with variance proxy $\sigma^2/(k+1)$. Hence
\[
  \Prob\bigl(|\epb_k^{\mathrm{new}}|>\tau_k\bigr)
  \le
  2\exp\left(-\frac{(k+1)\tau_k^2}{2\sigma^2}\right)
  \le 2(k+1)^{-c^2/2}.
\]
Because $c>2\sqrt{2}$, this bound is summable in $k$, so the first
Borel--Cantelli lemma \mbox{\cite[p.~188]{feller1957}} gives
$|\epb_k^{\mathrm{new}}|\le\tau_k$ a.s.\ for all sufficiently large
$k$. Since $\tau_k\to0$, this bound and
\eqref{eq:comparison-final} imply
$\epb_k^{\mathrm{new}}\to0$ and $D_k\to0$. Therefore,
\[
  \epb_k^{\mathrm{cur}}
    =\epb_k^{\mathrm{new}}-D_k\to0
  \quad\text{a.s.}
\]

\subsubsection*{Part V: Eventual monotonicity}
If the acceptance test holds, then the identity established in Part~I
gives $\Delta_k+D_k<-\tau_k$ and
$\thh_{k+1}=\btheta_{\text{new}}(k+1)$.
For $k \ge J_0(\omega)$, Part~IV gives
$|D_k| \le \tau_k$, hence
\[
  \Delta_k < -\tau_k-D_k \le 0
  \quad \text{a.s.}
\]
Along iterations on which a candidate is rejected,
$\thh_{k+1} = \thh_k$ and\\
$L(\thh_{k+1}) = L(\thh_k)$ trivially.
Therefore, on the event of probability one defining $J_0(\omega)$,
$L(\thh_{k+1})\le L(\thh_k)$ for every $k\ge J_0(\omega)$.

\subsubsection*{Part VI: Infinitely many candidates near the optimum}
Fix $\eta > 0$ and define
\[
  S_\eta := \{\btheta \in \Theta : L(\btheta) < L(\btheta^*) + \eta\},
\]
and analogously $S_{\eta/2}$ with $\eta$ replaced by $\eta/2$. Note
$S_{\eta/2} \subset S_\eta$. By the reachability assumption,
\[
  \Prob\bigl(\btheta_{\text{new}}(k+1) \in S_{\eta/2}\bigr)
    \ge \delta(\eta/2) > 0
  \quad \text{for every } k \ge k_0.
\]
Define the independent events
$G_k := \{\btheta_{\text{new}}(k+1) \in S_{\eta/2}\}$.\\ Then
$\sum_{k} \Prob(G_k) = \infty$, and the second Borel--Cantelli lemma
\mbox{\cite[p.~188]{feller1957}} gives
\begin{equation}
  \Prob(G_k \text{ infinitely often}) = 1,
  \label{eq:bc}
\end{equation}
i.e., candidates in $S_{\eta/2}$ are generated infinitely many times
almost surely.

\subsubsection*{Part VII: Eventual acceptance}
Since $\tau_k\to0$, Part~IV gives, almost surely,
for all sufficiently large $k$ (and hence beyond $J_0(\omega)$),
\begin{equation}
  |D_k| \le \eta/4
  \quad \text{and} \quad
  \tau_k < \eta/4.
  \label{eq:k1}
\end{equation}

{\widowpenalty=10000
On the event $\{\thh_k \notin S_\eta\} \cap G_k$, the definitions give
$L(\thh_k) \ge L(\btheta^*) + \eta$ and
$L(\btheta_{\text{new}}(k+1)) < L(\btheta^*) + \eta/2$, so
$\Delta_k < -\eta/2$. Combining with \eqref{eq:k1},\par}
\begin{multline*}
  \ybar(\btheta_{\text{new}}(k+1)) - \ybar(\thh_k)
    = \Delta_k + D_k \\
    < -\eta/2 + \eta/4
    = -\eta/4
    < -\tau_k.
\end{multline*}
The acceptance test is satisfied, hence
$\thh_{k+1} = \btheta_{\text{new}}(k+1) \in S_{\eta/2} \subset S_\eta$.

By \eqref{eq:bc}, almost every sample path has infinitely many indices
$k$ at which $G_k$ occurs and \eqref{eq:k1} holds. If
$\thh_k \notin S_\eta$ at any such index, the argument above gives
$\thh_{k+1} \in S_\eta$. Once
$\thh_k \in S_\eta$, Part~V's monotonicity ensures every subsequent
iterate also lies in $S_\eta$. Therefore
\[
  \Prob\bigl(\thh_k \in S_\eta \text{ for all sufficiently large }
  k\bigr) = 1.
\]

\subsubsection*{Part VIII: Convergence}
{\widowpenalty=10000
Applying Part~VII to $\eta=1/m$, $m\in\mathbb{N}$, and intersecting the
resulting events of probability one, the sample path argument based on
monotonicity used in the proof of \cite[Thm.~2.1]{spall2003} gives
$L(\thh_k)\to L(\btheta^*)$ a.s.
Assumption~\ref{ass:separation} then gives
$\thh_k\to\btheta^*$ a.s. Finally,
$\ybar(\thh_k)=L(\thh_k)+\epb_k^{\mathrm{cur}}$ and Part~IV give
$\ybar(\thh_k)\to L(\btheta^*)$ a.s.\par}
\end{proof}

\section{Numerical Experiments}
\label{sec:numerics}

\subsection{Experimental Setup}

The experiments examine three questions at finite budgets: whether the
proposed method continues to improve as the measurement budget grows,
what is gained by reusing measurements at the current estimate, and what
happens when measurement noise is ignored.

We use two test problems adapted from Spall's Examples~2.1 and~6.6
\cite{spall2003}. The first is the noisy quadratic problem
$L(\btheta)=\btheta^\top\btheta$, with $p=2$, $\Theta=[1,3]^2$,
$\thh_0=(2,2)^\top$, $\btheta^*=(1,1)^\top$, and $\sigma=2$. The second
sets $p=5$ in the skewed quartic loss:
\[
 L(\btheta)=\sum_{i=1}^{p}\left[(\Bmat\btheta)_i^2
 +0.1(\Bmat\btheta)_i^3+0.01(\Bmat\btheta)_i^4\right],
\]
where $\Bmat$ is the upper triangular matrix of ones,
$\Theta=[-5,5]^5$, $\thh_0=(1,\ldots,1)^\top$, $\btheta^*=\mathbf{0}$, and
$\sigma=3$. In both problems, measurements have the form
$y(\btheta)=L(\btheta)+\epsilon$, where all measurement errors are
i.i.d.\ $N(0,\sigma^2)$. Within each search, candidate points are
i.i.d.\ uniform on $\Theta$.

We compare three approaches. The proposed method follows
Section~\ref{sec:algorithm} with
$k_0=1$ and $\tau_k=3\sigma\sqrt{\log(k+1)/k}$. The remeasurement
benchmark computes the averaged loss value at the candidate point in the same
way and uses the same threshold, but replaces the pooled average at the
current estimate by an average of $k$ new measurements
at each comparison; neither average is reused in the next comparison. The
naive method stores one noisy loss measurement at the current estimate,
obtains one noisy loss measurement at each candidate point, and accepts a
candidate only when its measured loss is lower than the stored value; no
positive threshold is used. Each noisy loss measurement counts once
toward the measurement budget. The rule using a single measurement intentionally
reproduces the original baseline based on blind random search without noise control,
providing a direct contrast with averaging and thresholding.

\FloatBarrier
\begin{figure}[!htbp]
\centering
\includegraphics[width=0.94\linewidth]
  {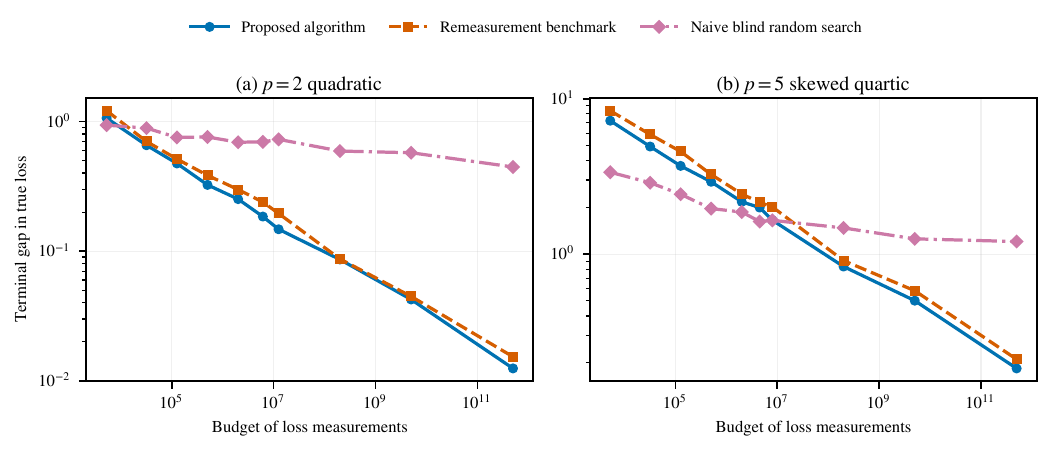}
\caption{Terminal gap in the true loss versus total measurement budget,
shown on logarithmic axes. Each marker is the median over $200$
independent runs for a selected value of the proposed method's comparison
index $K$; no individual run is selected for display. Connecting lines join
these medians across the reported values of $K$ and are not fitted rates.
Budgets are the median realized measurement counts of the proposed algorithm;
within each run, the comparison methods are stopped under the same realized
budget cap.}
\label{fig:budget-comparison}
\end{figure}

For each replication and each reported comparison index $K$ of the proposed method,
let $N_K$ denote the realized total number of noisy loss measurements consumed by
the proposed method through comparison $K$. This count includes the
$k_0$ initialization measurements, every measurement at a candidate point,
and one additional measurement at the current estimate after each rejection. In
the same replication, the remeasurement benchmark is run under that
realized budget and completes the largest number of full comparisons
whose cumulative measurement count does not exceed $N_K$; any unused
remainder is smaller than the cost of its next comparison. The naive
method uses exactly $N_K$ noisy loss measurements, including its
initialization measurement. Thus, within each replication, the methods
are matched by total measurement budget, not by the number of comparisons
or candidate points.

For each problem, we perform $200$ independent Monte Carlo replications.
For every reported comparison index $K$ of the proposed method, each replication
supplies a realized budget $N_K$ and a terminal gap for every method under that
matched budget. In Fig.~\ref{fig:budget-comparison}, the horizontal
coordinate is the median of the $200$ realized budgets, and the vertical
coordinate for each method is the median of its $200$ terminal gaps.
Thus the figure reports medians, not averages, and no individual
replication is selected. All search decisions use only noisy loss
measurements; the true loss is evaluated only to report the terminal gap
$L(\thh)-L(\btheta^*)$.

The Monte Carlo runs use NumPy 2.3.5 with the PCG64 generator and master
seed 20260727. Within each replication, whenever both the proposed method and
the remeasurement benchmark perform comparison $k$, they use the same candidate
point. Measurements at the current estimate for these methods and all random
variables for the naive method are generated from separate streams. Whenever
$n$ new independent Gaussian loss measurements are required at a point, we draw
a new $Z\sim N(0,1)$ and generate their average directly as
$L(\btheta)+\sigma Z/\sqrt{n}$. This random variable has exactly the distribution
of the average of $n$ independent $N(L(\btheta),\sigma^2)$ measurements. The $Z$
variates used to generate different averages are independent, except that when
both methods perform comparison $k$, the same $Z$ variate is used to generate
the averaged loss value at the shared candidate point. When the proposed method
accepts a candidate, the averaged loss value at the candidate point becomes the
pooled mean at the current estimate and is subsequently reused as specified by
the algorithm.
After a rejection, the additional measurement at the current estimate uses a
new independent error. Although the $n$ individual measurements represented by
an average are not generated explicitly, all $n$ are charged to $N_K$.

For the quadratic problem, the reported $K$ are 100, 250, 500, 1000,
2000, 3500, and 5000; for the quartic problem, they are 100, 250, 500,
1000, 2000, 3000, and 4000. Both continue at $K=2\times10^4$, $10^5$,
and $10^6$. With $k_0=1$, the proposed method uses $k+1$
measurements at a candidate point at comparison $k$, in addition to its one
initialization measurement and one measurement at the current estimate after each
rejection. Therefore,
\[
  N_K=1+\sum_{k=1}^{K}(k+1)+R_K
     =1+\frac{K(K+3)}{2}+R_K,
\]
where $R_K$ is the number of rejections among the first $K$
comparisons. For example, $N_{100}=5151+R_{100}$, so
$5151\le N_{100}\le5251$, consistent with the first plotted budget of
approximately $5.25\times10^3$ measurements. For the three additional
comparison indices, the proposed method and the remeasurement benchmark are
simulated directly. For the naive method, nested
running minima are generated from the corresponding distributions of
the order statistics, with the loss distribution for a single candidate
evaluated numerically.

\subsection{Results and Discussion}
\label{subsec:pooling-comparison}

Figure~\ref{fig:budget-comparison} shows the median values of the terminal
gaps. For the $p=2$ quadratic problem, the gaps for the proposed method and
the remeasurement benchmark decrease at similar rates. Both become much smaller
than the gap for naive blind random search as the budget grows. The two controlled methods share
averaging and thresholding, which explains their much faster decrease.
Their performance remains close, but the proposed algorithm generally
performs better. At the largest budget, the median gaps are $0.0125$ for
the proposed algorithm, $0.0154$ for the remeasurement benchmark, and $0.446$
for naive blind random search.

\Needspace{5\baselineskip}
For the quartic problem with $p=5$, naive blind random search initially
performs better at small budgets because it evaluates each candidate only
once and can therefore explore more broadly. The smallest
reported budget is about $5.25\times 10^3$ measurements.
At this budget, the proposed method, the remeasurement benchmark, and the
naive method have tested $100$, $71$, and about $5{,}250$ candidate points,
respectively.
The more candidates are tested, the more likely the smallest measured
loss is to be driven by an unusually negative noise realization rather
than a correspondingly small true loss. The naive method keeps this
misleadingly low value, which later candidates rarely beat; it therefore
plateaus while both controlled methods continue to improve and overtake
it. The gaps for the proposed method and the remeasurement benchmark again
remain close, with the proposed algorithm generally attaining the lower
gap. At the largest budget, their median gaps are $0.184$ and $0.210$,
compared with $1.20$ for naive blind random search. Relative to the initial
gaps ($6$ for the quadratic problem, $87.3$ for the quartic), the
proposed algorithm reduces the median terminal gap in the true loss by
$99.79\%$ in both problems. The remeasurement benchmark gives comparable reductions of
$99.74\%$ and $99.76\%$, whereas naive blind random search gives $92.57\%$
and $98.63\%$, respectively.

Across the $400$ runs at the largest tested budget, none of the $2801$
accepted updates increased the true loss. This empirical result is
consistent with the eventual monotonicity established in the convergence
analysis, but it is not a convergence proof or rate estimate.

\FloatBarrier
\section{Conclusions}

{\widowpenalty=10000
This paper makes two contributions to blind random search with noisy loss
measurements. First, it combines increasing averaging with a vanishing
positive threshold in the decision criterion. We prove that the
resulting search converges almost surely to the global minimizer.
Their combination ensures that the true loss is eventually nonincreasing
almost surely, despite the measurement noise.\par}

Second, the algorithm compares an averaged loss value at the candidate point,
based on new measurements, with a pooled average at the current estimate,
rather than
forming a new average at the current estimate for every comparison. Pooling
reduces repeated measurements at the current estimate, but it also creates
selection bias. The proof resolves this dependence by obtaining a uniform
bound over all possible current noise averages before the acceptance history
selects the realized one.

Under matched total measurement budgets, the numerical experiments show
that pooled reuse performs similarly to the remeasurement benchmark and
generally somewhat better than it. These results demonstrate the advantage of
pooling and motivate the treatment of its associated selection bias in the
convergence analysis.

Future work will investigate more efficient schedules for the running
average and threshold. It will also examine whether pooled comparisons
and the associated control of selection bias can be extended to localized
random search, simulated annealing, genetic algorithms, and other
stochastic search methods.

\end{document}